\documentclass{article}[dvipdfmx]

\usepackage[utf8]{inputenc}
\usepackage{url}
\usepackage{amsmath}
\usepackage{amsthm}
\usepackage{amssymb}
\usepackage{here}
\usepackage{caption}

\usepackage{cleveref}

\theoremstyle{definition}
\newtheorem{defi}{Definition}[section]

\theoremstyle{plain}
\newtheorem{prop}{Proposition}[section]
\newtheorem{theo}[prop]{Theorem}
\newtheorem{cor}[prop]{Corollary}
\newtheorem{lem}[prop]{Lemma}
\newtheorem{fac}[prop]{Fact}

\newtheorem*{cla}{Claim}
\newtheorem*{subcla}{Subclaim}

\newtheorem{maintheo}{Main Theorem}
\newtheorem{que}{Question}

\theoremstyle{remark}

\newcommand{\ZFC}{\mathrm{ZFC}}
\newcommand{\GCH}{\mathrm{GCH}}
\newcommand{\cf}[1]{\mathrm{cf}\left({#1}\right)}

\newcommand{\dmu}{\mathfrak{d}_{\mu}}

\newcommand{\domega}{\mathfrak{d}}

\newcommand{\smu}{\mathfrak{s}_{\mu}}
\newcommand{\salephomega}{\mathfrak{s}_{\aleph_\omega}}

\newcommand{\otp}[1]{\mathrm{otp} \left( {#1}  \right)}

\newcommand{\lkakko}{\left\lbrace}
\newcommand{\rkakko}{\right\rbrace}
\newcommand{\F}{\mathcal{F}}

\newcommand{\add}{\mathrm{add}}

\newcommand{\relmiddle}[1]{\mathrel{}\middle#1\mathrel{}}

\usepackage{xcolor}

\usepackage{tikz}

\usetikzlibrary{arrows}

\newcommand{\R}{\mathbb{R}}

\newcommand{\I}{\mathcal{I}}
\newcommand{\J}{\mathcal{J}}
\newcommand{\K}{\mathcal{K}}
\newcommand{\X}{\mathcal{X}}
\newcommand{\Y}{\mathcal{Y}}

\crefname{prop}{Proposition}{Propositions}
\crefname{theo}{Theorem}{Theorems}
\crefname{cor}{Corollary}{Corollaries}
\crefname{lem}{Lemma}{Lemmata}
\crefname{fac}{Fact}{Facts}
\crefname{prob}{Problem}{Problems}

\newcommand{\pposet}{\mathbb{P}}
\newcommand{\qposet}{\mathbb{Q}}
\newcommand{\rposet}{\mathbb{R}}

\usepackage[affil-it]{authblk}

\title{Dominating numbers at singular cardinals}
\author{Yusuke Hayashi\thanks{Supported by JST SPRING, Japan Grant Number JPMJSP2148}}
\affil{Graduate School of System Informatics, Kobe University, 1-1 Rokkodai, Nada-ku, 657-8501 Kobe, Japan. E-mail: 219x504x@stu.kobe-u.ac.jp}
\date{\today}

\begin{document}

\maketitle

\begin{abstract}
We study the generalized dominating number $\dmu$ at a singular cardinal $\mu$ of cofinality $\kappa$. We prove two basic lower bounds: in $\ZFC$, $\cf{[\mu]^\kappa,\subseteq} \leq \dmu$, and under mild cardinal-arithmetic assumptions, $2^{<\mu} \leq \dmu$. We also clarify when $\dmu$ can differ from $2^\mu$: assuming $\GCH$ and $\kappa = \cf{\mu} > \omega$, a finite-support iteration of Cohen forcing of length $\mu^{++}$ yields $\dmu < 2^\mu$. On the other hand, for $\kappa = \cf{\mu} = \omega$, natural $\mu$-cc posets force $\dmu = 2^\mu$. 
\end{abstract}

\section{Introduction}
Cardinal invariants are cardinals that describe combinatorial properties of mathematical objects related to the real numbers. So far, many cardinal invariants have been studied extensively by many set theorists. For example, let $\add(\mathcal N)$ be the least cardinal $\kappa$ such that the union of $\kappa$ many Lebesgue null sets is no longer null. By the $\sigma$-additivity of the Lebesgue measure, $\add(\mathcal N)$ cannot be $\aleph_0$. In addition, since $\mathbb R$ is the union of all singletons, we have $\add(\mathcal N)\le 2^{\aleph_0}$. Thus cardinal invariants typically take values between $\aleph_1$ and $2^{\aleph_0}$. There are many results that describe which inequalities among these invariants are provable and which are independent of $\ZFC$ (see \cite{blass2009combinatorial,bartoszynski1995set}). One aim of studying cardinal invariants is to understand relationships among mathematical objects via inequalities between these invariants.

Many classical cardinal invariants are defined using the structure of $P(\omega)$ or $\omega^{\omega}$. For any uncountable cardinal $\mu$, we can naturally lift these definitions by considering analogous properties of $P(\mu)$ or $\mu^{\mu}$. Research on these higher cardinal invariants has developed in recent years. For example, the case of regular $\mu$ has been studied in \cite{cummings1995cardinal,brendle2018cichon}, and the singular cardinal case has also been explored in \cite{garti2012ultrafilter,garti2020cardinal}. In particular, the dominating number at singular cardinals has been studied by Garti and Shelah \cite{garti2018dear,shelah2020dmu}. The following is a theorem of  Shelah \cite{shelah2020dmu}. \textbf{In this section, suppose that $\mu$ is a singular cardinal with cofinality $\kappa$.} We write $\dmu$ for the generalized dominating number at singular $\mu$ (for the precise definition, see \Cref{defi:generalized dominating}).

\begin{theo}[Shelah \cite{shelah2020dmu}]\label{theo:shelahs estimation of dmu}
    If $\alpha^\kappa < \mu$ for all $\alpha < \mu$, then $\dmu = 2^\mu$.
\end{theo}

It is natural to ask whether the cardinal arithmetical assumption is necessary in Theorem 1.1. Shelah \cite[Claim~3.1]{shelah2020dmu} asserts that the assumption is necessary by showing that if we add $\mu^{++}$-many $\kappa^+$-Cohen sets over a model of GCH, then we obtain a model in which $\dmu < 2^\mu$. But this contradicts \Cref{theo:shelahs estimation of dmu} since the assumption of \Cref{theo:shelahs estimation of dmu} holds in the forcing extension. This question is still open as far as the author knows.

Towards its solution, this paper investigates $\dmu$ with more details. In particular, we give two lower bounds of $\dmu$ under some mild assumptions, and discuss the consistency of $\dmu < 2^\mu$.

By \Cref{theo:shelahs estimation of dmu}, if $\alpha^\kappa < \mu$ holds for every $\alpha < \mu$, then $\dmu = 2^\mu$ follows immediately. In $\ZFC$, the equality $\alpha^\kappa = 2^\kappa \cdot \cf{[\alpha]^\kappa, \subseteq}$ holds. Main Theorem 1 shows that, under appropriate assumptions on $2^\kappa$ and $\cf{[\alpha]^\kappa, \subseteq}$, even when $\alpha^\kappa \geq \mu$ for some $\alpha < \mu$, $2^{<\mu}$ can still be a lower bound for $\dmu$.

\begin{maintheo}[see Theorem \ref{theo:main dmu >= 2^mu}]
    Suppose that 
    \begin{itemize}
        \item $2^\kappa < \mu^{+\kappa^+}$,
        \item $\cf{[\alpha]^\kappa, \subseteq} <\mu$ for all $\alpha < \mu$, and
        \item $2^\kappa < 2^\nu$ for some $\nu < \mu$.
    \end{itemize}
    Then $\dmu \geq 2^{<\mu}$.
\end{maintheo}

We note that the second assumption above is very weak. Indeed, it is known that its negation implies the existence of $0^{\sharp}$. A key in the proof of Main Theorem 1 is the use of almost distinct families of functions (a.d.f.). This is based on the methods used to analyze the dominating number of $\omega^{\omega_1}$ by Jech and Prikry in \cite{jech1984cofinality}.

Main Theorem 2 gives a lower bound for $\dmu$ without any additional assumptions beyond $\ZFC$.

\begin{maintheo}[see Theorem \ref{theo:cfkappa^lambdasub<=dmu}]
    $\cf{[\mu]^\kappa, \subseteq} \leq \dmu$.
\end{maintheo}

$\cf{[\mu]^\kappa, \subseteq}$ is one of the most important cardinal invariants in the combinatorics of $\mu$. For example, $\mu^\kappa = \cf{[\mu]^\kappa, \subseteq} \cdot 2^\kappa$ holds in $\ZFC$.

Also, Main Theorem 2 can be applied to the investigation of the analogue of $\mathfrak{s}\leq\domega$. It is well known that $\mathfrak{s}\leq\domega$ is provable in $\ZFC$, and interval partitions play a central role in the proof. However, the behavior of interval partitions on singular cardinals is different from that on regular cardinals. Hence, extending the argument of the proof of $\mathfrak{s} \leq \domega$ is difficult. On the other hand, for the higher analogue of the splitting number, Zapletal \cite{zapletal1997splitting} has proved the following.

\begin{fac}[Zapletal \cite{zapletal1997splitting}]
    $\smu \leq \cf{[\mu]^\kappa, \subseteq}$.
\end{fac}

By combining these two results, we obtain $\smu \leq \dmu$ for every singular $\mu$.

Main Theorems 3 and 4 focus on the consistency of $\dmu<2^\mu$. Main Theorem 3 shows that when $\mu$ has uncountable cofinality, the consistency follows easily.

\begin{maintheo}[see Theorem \ref{theo:main3}]
    Assume $\GCH$ and $\kappa\neq \omega$. Let $\pposet$ be a finite support iteration of the Cohen forcing $2^{< \omega}$ of length $\mu^{++}$. Then $\Vdash_{\pposet} \dmu < 2^{\mu}$.
\end{maintheo}

In contrast, Main Theorem 4 shows that if $\mu$ has countable cofinality, it is hard to obtain the consistency of $\dmu < 2^\mu$ with $\mu$-cc forcings.

\begin{maintheo}[see Theorem \ref{theo:main4}]
    Assume $\GCH$ and $\kappa = \omega$. Let $\lambda > \mu$ be a regular cardinal, and $\pposet$ be an atomless complete Boolean algebra with $|\pposet| = \lambda$. If $\pposet$ has $\mu$-cc and $\Vdash_\qposet \pposet / \dot{G} \text{ is non-trivial}$  for every complete subalgebra $\qposet\subseteq \pposet$ with $|\qposet| < |\pposet|$, then $\Vdash_{\pposet} \dmu =  \lambda = 2^\mu$.
\end{maintheo}

\section{Preliminaries}
First, we review the definition of the dominating number at singular cardinals. The more detailed basic properties can be found in the work of Shelah and Garti\cite{garti2018dear,shelah2020dmu}.

\begin{defi}\label{defi:generalized dominating}
    Let $\mu$ be a singular cardinal with cofinality $\kappa$.
    \begin{itemize}
        \item For $f, g \in \kappa^\mu$, we say that $f$ dominates $g$ if $\{ \alpha < \mu \mid f(\alpha) \leq g(\alpha) \}$ has size less than $\mu$.
        \item For $\mathcal{D}\subseteq \kappa^\mu$, we call $\mathcal{D}$ \textit{dominating family} if for any $f \in \kappa^\mu$, there is some $g \in \mathcal{D}$ such that $g$ dominates $f$.
        \item We define
        \[ \dmu = \min \lkakko |\mathcal{D}| \mid \mathcal{D} \text{ is a dominating family.} \rkakko \]
    \end{itemize}
\end{defi}

Next, we show the basics of almost distinct functions, which play a important role in this paper.

\begin{defi}
    Let $A$ be a set, and let $\kappa$ be cardinal. Let $\mathcal{I}$ be an ideal on $A$. Two functions $f, g \in \kappa^A$ are called \textit{$\mathcal{I}$-almost distinct} or \textit{$\mathcal{I}$-a.d.} if $\lkakko a \in A \mid f(a) = g(a) \rkakko \in \mathcal{I}$ holds.

    A family $\mathcal{F}\subseteq \kappa^A$ is an \textit{$\I$-almost distinct family of functions} or \textit{$\I$-a.d.f.} if the functions in $\mathcal{F}$ are pairwise $\I$-a.d.
\end{defi}

\begin{defi}
\begin{enumerate}
    \item For an ordinal $\alpha$ and a cardinal $\kappa$, define \textit{the bounded ideal on }$[\alpha]^\kappa$ by $\lkakko A \subseteq [\alpha]^\kappa \mid \exists x \in [\alpha]^\kappa \forall y \in A (x \nsubseteq y) \rkakko$.
    \item Let $A, B$ be sets. For ideals $\I$ and $\J$ on $A$ and $B$, we define \textit{the Fubini product ideal} $\I\times \J$ on $A \times B$ by
    \[ X \in \I\times \J \Longleftrightarrow \lkakko a \in A \mid  X_a \notin \J \rkakko \in \I, \]
    where $X_a = \lkakko b \in B \mid (a, b) \in X \rkakko$
    \item Let $\kappa$ be a cardinal. We say that an ideal $\I$ is $\kappa$\textit{-complete} if for every $\mu < \kappa$ and every sequence $\langle X_i \in  \I \mid i < \mu \rangle$, $\bigcup_{i < \mu} X_i \in \I$.
\end{enumerate}
\end{defi}

\begin{lem}\label{lem:transitivity of completeness}
    Let $\kappa \leq \lambda$ be cardinals and let $A, B$ be sets. Suppose
    \begin{itemize}
        \item $\I$ is a $\kappa$-complete ideal, and
        \item $\J$ is a $\lambda$-complete ideal.
    \end{itemize}
    Then $\I\times \J$ is a $\kappa$-complete ideal.
\end{lem}
\begin{proof}
    Let $\mu < \kappa$. Take a sequence $\langle X_i \mid i < \mu \rangle$ in $\I\times \J$. By the definition of $\I\times \J$, for any $i < \mu$, $\lkakko a \in A \mid (X_i)_a \notin \J \rkakko \in \I$. Hence
    \[ \lkakko a \in A \relmiddle{|} \left(\bigcup_{i< \mu} X_i\right)_a \notin \J \rkakko = \bigcup_{i< \mu}\lkakko a \in A \relmiddle{|} \left( X_i\right)_a \notin \J \rkakko \in \I. \]
    This completes the proof.
\end{proof}

\section{Almost disjoint and Dominating numbers}
\textbf{From now on, again, suppose that $\mu$ is a singular cardinal with cofinality $\kappa$.}

In this section, we construct some a.d.f.s. First, we show that constructing a large a.d.f. immediately yields a lower bound for $\dmu$.

\begin{theo}\label{theo:computing dmu from adf}
    Let $\delta \leq \mu$ be an ordinal and $\I$ be a $\kappa^+$-complete proper ideal on $\delta$. If there is an $\I$-a.d.f. $\mathcal{F} \subseteq \kappa ^\delta$, then $\dmu \geq |\mathcal{F}|$.
\end{theo}

\begin{proof}
    Let $\mathcal{D}\subseteq \kappa^\delta$ be a dominating family with $|\mathcal{D}| = \dmu$ and $\sigma: \mu \rightarrow \delta$ be a surjection. For $f \in \mathcal{F}$, Define $\tilde{f} \in \kappa^\mu$ by $\tilde{f}(\alpha) = f(\sigma(\alpha))$ for all $\alpha< \mu$. For each $d\in \mathcal{D}$, let $\mathcal{F}_d = \{ f\in \mathcal{F} \mid \tilde{f} \leq d \}$. Since $\mathcal{D}$ be a dominating family, $\mathcal{F} = \bigcup_{d\in \mathcal{D}}\mathcal{F}_{d}$ holds. 
    \begin{cla}
        $|\mathcal{F}_d| < \kappa$ for all $d\in \mathcal{D}$.
    \end{cla}
    \begin{proof}[Proof of claim]
        If not, there exists $d \in \mathcal{D}$ such that $\mathcal{F}_d$ has size at least $\kappa$. Take any subset $\mathcal{F}'\subseteq \mathcal{F}_d$ with $|\mathcal{F}'| =\kappa$. Since $\mathcal{F}$ is an $\I$-a.d.f, $\lkakko \eta <\delta \mid f_0(\eta) = f_1(\eta) \rkakko$ is in $\I$ for every $f_0 \neq f_1$ in $\mathcal{F}'$. 
        Then $A = \bigcup_{f_0, f_1 \in \mathcal{F}', f_0 \neq f_1} \lkakko \eta < \delta \mid f_0(\eta) = f_1(\eta) \rkakko \in \I$ since $\I$ is $\kappa^+$-complete. Since $\I$ is proper and $\sigma$ is a surjection, we can find $\alpha < \mu$ such that  $\sigma(\alpha) \notin A$. By the choice of $\alpha$, $\tilde{f_0}(\alpha) \neq \tilde{f_1}(\alpha)$ for all $f_0, f_1\in \mathcal{F}'$ with $f_0 \neq f_1$. Therefore, $\lkakko \tilde{f}(\alpha) \mid f\in \mathcal{F}_d \rkakko \supseteq \lkakko \tilde{f}(\alpha) \mid f\in \mathcal{F}' \rkakko$ is unbounded in $\kappa$. However, by the definition of $\mathcal{F}_d$, $\tilde{f}(\alpha) \leq d(\alpha)$ for all $f \in \mathcal{F}_d$. This is a contradiction.
    \end{proof}
    By the claim above, $|\mathcal{F}| = \left| \bigcup_{d\in \mathcal{D}} \mathcal{F}_d \right| \leq |\mathcal{D}| \cdot \kappa =\dmu$.
\end{proof}

Hence, to obtain a lower bound for $\dmu$, it suffices to construct an a.d.f. of large cardinality.
The following proposition describes how to combine a.d.f.s to produce a new a.d.f.

\begin{prop}\label{prop:combine adfs}
    Let $\lambda_0, \lambda_1, \lambda_2$, and $\lambda_3$ be cardinals with $\lambda_1 \leq \lambda_2$, and let $\I, \J$ be ideals on $A, B$, respectively. Suppose
    \begin{itemize}
        \item $\mathcal{F}\subseteq \lambda_1^A$ be an $\I$-a.d.f. with $|\mathcal{F}| = \lambda_0$, and
        \item $\mathcal{G}\subseteq \lambda_3^B$ be a $\J$-a.d.f. with $|\mathcal{G}| = \lambda_2$.
    \end{itemize}
    Then there is an $\I \times \J$-a.d.f. $\mathcal{H}\subseteq \lambda_3^{A\times B}$ such that $|\mathcal{H}| = \lambda_0$.
\end{prop}
\begin{proof}
    Since $\lambda_1 \leq \lambda_2 = |\mathcal{G}|$, we can take an injection $\pi$ from $\lambda_1$ to $\mathcal{G}$. For $f \in \mathcal{F}$, define $h_f \in \lambda_3^{A\times B}$ by $h_f(a, b) = \pi(f(a))(b)$. We show that $\mathcal{H} = \lkakko h_f \mid f\in \mathcal{F} \rkakko$ is as desired. Let $f, f' \in \mathcal{F}$ and $f\neq f'$. It suffices to show that $X = \lkakko (a,b) \in A\times B \mid h_f(a, b) = h_{f'} (a, b) \rkakko \in \I \times \J$. Since $\pi$ is an injection and $\mathcal{G}$ is a $\J$-a.d.f, if $f(a) \neq f'(a)$, then $X_a \in \mathcal{J}$. Therefore, 
    \[ \lkakko a \in A \mid  X_a \notin \J \rkakko \subseteq \lkakko a \in A \mid f(a) = f'(a) \rkakko \in \I \]
    and this completes the proof.
\end{proof}

Let $A$ be a set and $\I$ be an ideal on $A$. For cardinals $\lambda_0$ and $\lambda_1$, regularly, we shall represent an a.d.f. $\mathcal{F}\subseteq \lambda_1^A$ with $|\mathcal{F}|=\lambda_0$ in the table below.

    \begin{table}[H]
    \begin{center}
\begin{tabular}{c|cccc}
              & size        & dom & ran         & ideal \\ \hline
$\mathcal{F}$ & $\lambda_0$ & $A$ & $\lambda_1$ & $\I$ 
\end{tabular}
\end{center}
\end{table}

Then \Cref{prop:combine adfs} asserts the following rule for transforming such tables.

\begin{table}[H]
\begin{center}
    \begin{tabular}{c|cccc}
              & size        & dom & ran         & ideal \\ \hline
$\mathcal{F}$ & $\lambda_0$ & $A$ & $\lambda_1$ & $\I$  \\
$\mathcal{G}$ & $\lambda_2( \geq \lambda_1)$ & $B$ & $\lambda_3$ & $\J$  \\ \hline \hline
$\mathcal{H}$ & $\lambda_0$ & $A \times B$ & $\lambda_3$ & $\I \times \J$  \\
\end{tabular}
\end{center}
\end{table}

Before constructing an a.d.f, we show the $\kappa^+$-completeness of the bounded ideal on $([\alpha]^\kappa, \subseteq)$. 

\begin{lem}\label{lem:completenes of the bounded ideal}
    Let $\alpha \geq \kappa^+$ and $C$ be a cofinal subset in $([\alpha]^\kappa, \subseteq)$. Let $\I$ be the restriction of the bounded ideal on $([\alpha]^\kappa, \subseteq)$ to $C$. Then $\I$ is $\kappa^+$-complete.
\end{lem}

\begin{proof}
    Note that $\I = \lkakko A\subseteq C \mid \exists  x \in  [\alpha]^\kappa \forall
    y \in A (x \nsubseteq y) \rkakko$. Let $\lkakko A_i \mid i < \kappa \rkakko$ be a subset of $\I$. It suffices to show that $A = \bigcup_{i < \kappa} A_i \in \I$. For each $i < \kappa$, by the definition of the bounded ideal, we can find $x_i \in [\alpha]^\kappa$ such that $x_i \nsubseteq y$ for all $y \in A_i$. Let $x= \bigcup _{i < \kappa}x_i$. Then for all $i < \kappa$ and $y \in A_i$, $x_i\subseteq x$, so $x \nsubseteq y$. Therefore $x$ witnesses $A \in \I$.
\end{proof}

In Theorems \ref{theo:existence of large adf} and \ref{theo:adf filling gap} below, we construct the a.d.f.s required for the main theorem.

\begin{theo}\label{theo:existence of large adf}
    Let $\alpha < \mu$ be an ordinal. Suppose that 
    \begin{itemize}
        \item $2^\kappa < \mu^{+\kappa^+}$,
        \item $\cf{[\alpha]^\kappa, \subseteq} < \mu$, and
        \item $2^\kappa < 2^\alpha$.
    \end{itemize}
    Let $C$ be a cofinal subset in $([\alpha]^\kappa , \subseteq)$ with $|C| < \mu$, and $\I$ be the restriction of the bounded ideal on $[\alpha]^\kappa$ to $C$. Then for any regular cardinal $\lambda \in (2^\kappa, 2^\alpha]$, there are an ordinal $\delta<\mu$ and an $\I$-a.d.f. $\mathcal{F} \subseteq \delta^C$ such that $|\mathcal{F}| = \lambda$.

        \begin{table}[H]
    \begin{center}
\begin{tabular}{c|cccc}
              & size        & dom & ran         & ideal \\ \hline
$\mathcal{F}$ & $\lambda ( \in (2^\kappa, 2^\alpha])$ & $C$ & $\delta(<\mu)$ & $\I$ 
\end{tabular}
\end{center}
\end{table}
\end{theo}

\begin{proof}
    Let $\X$ be a subset of $\mathcal{P}(\alpha)$ with $|\X| = \lambda$. For $X\in \X$ and $x \in C$, define $f_X(x) = X \cap x$. 
    \begin{cla}
        $\lkakko f_X \mid X \in \X \rkakko$ is an $\I$-a.d.f.
    \end{cla}
    \begin{proof}[Proof of claim]
        Let $X, Y \in \X$ and $X\neq Y$. It is sufficient to show that $\lkakko x \in C \mid f_X(x) = f_Y(x) \rkakko \in \I$. Fix $\beta \in X \triangle Y$. If $\beta \in x$, then $f_X(x) = X\cap x \neq Y\cap x =f_Y(x)$. Therefore, 
        \[ \lkakko x \in C \mid f_X(x) = f_Y(x) \rkakko \subseteq  \lkakko x \in C \mid \lkakko \beta \rkakko \nsubseteq x \rkakko \in \I.\]
    \end{proof}
    For $\Y \subseteq \X$ and a cardinal $\nu$, we call $\Y$ $\nu$\textit{-good} if $|\lkakko f_X(x) \mid X \in \Y \rkakko| \leq \nu$ for all $x \in C$. Clearly, $\X$ is $2^\kappa$-good.
    \begin{cla}
        For $\nu \in [\mu, 2^\kappa]$, if $\Y \subseteq \X$ satisfies that $|\Y|=\lambda$ and $\Y$ is $\nu$-good, then there is some $\Y'\subseteq \Y$ and a cardinal $\nu'<\nu$ such that $|\Y'| = \lambda$ and $\Y'$ is $\nu'$-good.
    \end{cla}
    \begin{proof}[Proof of claim]
        (Case 1: $\nu$ is regular) For any $x\in C$, take a sequence $\langle A_x^\xi \mid \xi < \nu \rangle$ such that
        \begin{itemize}
            \item $\bigcup_{\xi < \nu}A_x^\xi = \lkakko f_X(x)\mid X \in \Y \rkakko$, and
            \item if $\xi < \zeta$, then $A_x^\xi \subseteq A_x^\zeta$, and
            \item for every $\xi < \nu$, $|A_x^\xi| \leq |\xi|$.
        \end{itemize}
        Then for any $X \in \Y$, we can choose $\xi_X <\nu$ such that $f_X(x) \in A_x^{\xi_X}$ for all $x \in C$ since $|C| < \mu \leq \nu$ and $\nu$ is regular. By the regularity of $\lambda$ and applying the pigeonhole principle for the map $X\mapsto \xi_X$, we obtain $\Y' \subseteq \Y$ and $\xi < \nu$ such that $|\Y'| = |\Y| = \lambda$ and $\xi_X =\xi$ for all $X \in \Y'$. Thus $\nu'=|\xi|$ is as desired.

        (Case 2: $\nu$ is singular) Since $2^\kappa < \mu^{+\kappa^+}$, $\cf{\nu} \leq \kappa$. Let $\langle \nu_\xi \mid \xi < \cf{\nu} \rangle$ be a cofinal increasing sequence in $\nu$. For any $x\in C$, take a sequence $\langle A_x^\xi \mid \xi < \cf{\nu} \rangle$ such that
        \begin{itemize}
            \item $\bigcup_{\xi < \nu}A_x^\xi = \lkakko f_X(x)\mid X \in \Y \rkakko$, and
            \item if $\xi < \zeta$, then $A_x^\xi \subseteq A_x^\zeta$, and
            \item for every $\xi < \nu$, $|A_x^\xi| \leq |\nu_\xi|$.
        \end{itemize}
        Let us define $C_\xi^X = \lkakko x \in C \mid  f_X(x) \in A_x^{\xi_X} \rkakko$ for all $\xi < \cf{\nu}$ and $X \in \X$.
        \begin{subcla}
            For every $X \in \Y$, there exists some $\xi_X < \cf{\nu}$ such that $C_{\xi_X}^X$ is cofinal in $([\alpha]^\kappa, \subseteq)$.
        \end{subcla}
        \begin{proof}[Proof of subclaim]
            For contradiction, suppose that $C_\xi^X$ is not cofinal for all $\xi < \cf{\nu}$. For $\xi < \cf{\nu}$, we can choose $x_\xi \in [\alpha]^\kappa$ such that $x_\xi \nsubseteq y$ for all $y \in C_\xi^X$. Then $x = \bigcup_{\xi < \cf{\nu}} x_\xi$ witnesses that $C = \bigcup_{\xi < \cf{\nu}} C_\xi^X$ is not cofinal. This is a contradiction.
        \end{proof}
        By the regularity of $\cf{\lambda}$ and applying the pigeonhole principle for the map $X\mapsto \xi_X$, we obtain $\Y' \subseteq \Y$ and $\xi < \cf{\nu}$ such that $|\Y'| = |\Y| = \lambda$ and $\xi_X =\xi$ for all $X \in \Y'$. Next, we show that $\Y'$ is $|C|\cdot \nu_\xi$-good. To prove this, it suffices to construct an injection $\pi_x$ from $\lkakko f_X(x) \mid X\in \Y' \rkakko$ to $\bigcup_{y \in C} \lkakko y \rkakko \times A_y^\xi$ for all $x \in C$. 
        
        Let $p$ be an element of $\lkakko f_X(x) \mid X\in \Y' \rkakko$. Take an $X\in \Y'$ such that $f_X(x) = p$. By the construction of $\Y' \subseteq \Y$, we can take $y \in C^X _\xi \subseteq C$ such that $x \subseteq y$ and $f_X(y) \in A_y^\xi$. Here, define $\pi_x(p) = (y, f_X(y))$. Let $X, Y \in \Y'$ and $p = f_X(x), q = f_Y(x)$, and assume that $\pi_x(p) = \pi_x (q)$. Then there is $y \in C$ with $x \subseteq y$ such that $(y, f_X(y)) = \pi_x(p) = \pi_x(q) = (y, f_Y(y))$. By the definition of the functions, this implies $X \cap y = f_X(x) = f_Y(y) = Y \cap y$. Since $x \subseteq y$, $p=f_X(x) = X\cap x = Y \cap x = f_Y(x) = q$. Therefore, $\pi_x$ is an injection.
    \end{proof}
    By applying the claim above successively, eventually, we obtain $\Y\subseteq \X$ and $\delta < \mu$ such that $\Y$ is $\delta$-good and $|\Y| = \lambda$. Therefore, $\F = \lkakko f_X \mid X \in \Y \rkakko$ is as desired.
\end{proof}

\begin{theo}\label{theo:adf filling gap}
    Let $\theta < \mu$ be a cardinal. Let $D$ be a cofinal subset in $([\theta]^\kappa, \subseteq)$, and $\J$ the restriction of the bounded ideal on $[\alpha]^\kappa$ to $D$. Then there is an $\J$-a.d.f. $\mathcal{G} \subseteq \kappa ^D$ such that $|\mathcal{G}| = \theta$.

            \begin{table}[H]
    \begin{center}
\begin{tabular}{c|cccc}
              & size        & dom & ran         & ideal \\ \hline
$\mathcal{G}$ & $\theta$ & $D$ & $\kappa$ & $\J$ 
\end{tabular}
\end{center}
\end{table}
\end{theo}

\begin{proof}
    For $x \in D$, let $\pi_x$ be an injection from $x$ to $\kappa$. For $\alpha< \theta$, define $g_\alpha \in \kappa^D$ by 
    \[ g_\alpha(x) = \begin{cases}
        \pi_x(\alpha) & \text{if $\alpha \in x$,} \\
        0 & \text{otherwise.}
    \end{cases} \]
    It suffices to show that if $\alpha, \beta < \theta$ and $\alpha \neq \beta$, then $g_\alpha$ and $g_\beta$ are $\J$-a.d. Take such $\alpha, \beta < \theta$. Then if $x \in D$ satisfies $x\supseteq \lkakko \alpha, \beta
    \rkakko$, then $g_\alpha(x) = \pi_x(\alpha) \neq \pi_x(\beta) = g_\beta(x)$. Hence
    \[ \lkakko x \in D \mid g_\alpha(x) = g_\beta(x) \rkakko \subseteq \lkakko x \in D \mid \lkakko \alpha, \beta \rkakko \nsubseteq x \rkakko \in \J. \]
    This completes the proof.
    \end{proof}

Applying \Cref{prop:combine adfs} to the two families of \Cref{theo:existence of large adf,theo:adf filling gap}, we get the following corollary:
\begin{cor}\label{cor:main adf}
    Let $\nu < \mu$ be a cardinal with $\kappa < \nu$. Suppose that 
    \begin{itemize}
        \item $2^\kappa < \mu^{+\kappa^+}$,
        \item $\cf{[\alpha]^\kappa, \subseteq} <\mu$ for all $\alpha < \mu$, and
        \item $2^\kappa < 2^\nu$.
    \end{itemize}
    Then for any $\alpha \in [\nu, \mu)$ and regular cardinal $\lambda \in (2^\kappa, 2^{\alpha}]$, there is a $\kappa^+$-complete ideal $\K$ on $S$ for some set $S$ with $|S| < \mu$, and a $\K$-a.d.f. $\mathcal{H} \subseteq \kappa ^S$ such that $|\mathcal{H}| = \lambda$.
\end{cor}

\begin{proof}
    By \Cref{theo:existence of large adf}, there is an $\I$-a.d.f. $\F \subseteq \delta^C$, where
    \begin{itemize}
        \item $C$ is cofinal subset in $([\alpha]^\kappa , \subseteq)$ with $|C| < \mu$,
        \item $\I$ is the restriction of the bounded ideal on $[\alpha]^\kappa$ to $C$,
        \item $\delta< \mu$, and
        \item $|\F|=\lambda$.
    \end{itemize}
    Take a cardinal $\theta \in [\delta, \mu)$. By \Cref{theo:adf filling gap}, there is a $\J$-a.d.f. $\mathcal{G }\subseteq \kappa^D$, where 
    \begin{itemize}
        \item $D$ is cofinal subset in $([\theta]^\kappa , \subseteq)$ with $|D| < \mu$,
        \item $\J$ is the restriction of the bounded ideal on $[\theta]^\kappa$ to $D$, and
        \item $|\mathcal{G}|=\theta$.
    \end{itemize}
    By applying \Cref{prop:combine adfs}, we obtain an $\I \times \J$-a.d.f. $\mathcal{H} \subseteq \kappa^{C\times D}$ such that $|\mathcal{H}| = \lambda$. Define $S = C \times D$ and $\K = \I \times \J$. By \Cref{lem:completenes of the bounded ideal,lem:transitivity of completeness}, $\I$, $\J$, and $\K$ are $\kappa^+$-complete. Also, $C$ and $D$ have size less than $\mu$. Then $S$ also has size less than $\mu$.

\begin{table}[H]
\begin{tabular}{c|cccc|c}
              & size        & dom & ran         & ideal &  \\ \hline
$\mathcal{F}$ & $\lambda ( \in (2^\kappa, 2^\alpha])$ & $C$ & $\delta(<\mu)$ & $\I$ & By \Cref{theo:existence of large adf}  \\
$\mathcal{G}$ & $\theta ( \geq \delta)$ & $D$ & $\kappa$ & $\J$  & By \Cref{theo:adf filling gap}  \\ \hline \hline
$\mathcal{H}$ & $\lambda$ & $C \times D$ & $\kappa$ & $\I \times\J$  & By \Cref{prop:combine adfs}
\end{tabular}
\end{table}
\end{proof}

By combining \Cref{cor:main adf,theo:computing dmu from adf}, we obtain the following.

\begin{theo}\label{theo:main dmu >= 2^mu}
    Suppose that 
    \begin{itemize}
        \item $2^\kappa < \mu^{+\kappa^+}$,
        \item $\cf{[\alpha]^\kappa, \subseteq} <\mu$ for all $\alpha < \mu$, and
        \item $2^\kappa < 2^\nu$ for some $\nu < \mu$.
    \end{itemize}
    Then $\dmu \geq 2^{<\mu}$.
\end{theo}
\begin{proof}
    It suffices to show that for all $\alpha \in [\nu, \mu)$, $2^\alpha \leq \dmu$. For any $\alpha \in [\nu, \mu)$, by \Cref{cor:main adf,theo:computing dmu from adf}, we obtain that $\lambda \leq \dmu$ for each regular $\lambda \in (2^\kappa, 2^\alpha]$. If $2^\alpha$ is regular, the case $\lambda = 2^\alpha$ implies $2^\alpha \leq \dmu$ directly. Also, if $2^\alpha$ is singular, we obtain $\lambda \leq \dmu$ for unboundedly many $\lambda < 2^\alpha$. Therefore, in either case, we obtain that $2^\alpha \leq \dmu$ and this completes the proof.
\end{proof}

\section{The cofinality of $\cf{[\mu]^\kappa, \subseteq}$ and Cardinal invariants}
    
In this section we establish a lower bound for $\dmu$ provable in $\ZFC$ and investigate relations with other cardinal invariants at $\mu$.

\begin{theo}\label{theo:cfkappa^lambdasub<=dmu}
    $\cf{[\mu]^\kappa, \subseteq} \leq \dmu$.
\end{theo}

\begin{proof}
    Let $E^\mu_\kappa = \lkakko \alpha
     < \mu \mid \cf{\alpha} = \kappa \rkakko$. For $\alpha \in E^\mu _{\kappa}$, take $\ell_\alpha\subseteq \alpha$ be a ladder, that is, $\otp{\ell_\alpha} = \kappa$ and $\ell_\alpha$ is cofinal in $\alpha$. Let $\pi : \mu \rightarrow \mu\times\mu$ be a bijection such that every element of $\mu\times \mu$ appears $\mu$ many times, and let $\alpha_\gamma$ (resp. $\beta_\gamma$) denote the 1st (resp. 2nd) coordinate of $\pi(\gamma)$.

    Let $C \subseteq [\mu]^\kappa$ be a cofinal subset of $[\mu]^\kappa$ with $|C|=\cf{[\mu]^\kappa, \subseteq}$. For $x \in C$, we define the function $f_x \in \kappa^\mu$ as follows:
    \[ f_x(\gamma) =\begin{cases}
        \otp{ \ell_{\alpha_\gamma} \cap \sup(x\cap \alpha_\gamma)} &\text{if $\alpha_\gamma \in E^\mu _\kappa$ and $\sup(x\cap \alpha_\gamma) < \alpha_\gamma$,} \\
        0 & \text{otherwise.}
    \end{cases} \]

    Let $\mathcal{D}\subseteq \kappa^\mu$ be a dominating family with $|\mathcal{D}|=\dmu$. Suppose, for contradiction, that $|\mathcal{D}| =\dmu < \cf{[\mu]^\kappa, \subseteq} =|C|$. We define $C_{d, \gamma}$ by $\{ x \in C \mid \forall \gamma' \geq \gamma [ f_x (\gamma') \leq d(\gamma') ] \}$.

    Here, if $|\bigcup C_{d, \gamma}| \leq \kappa$ holds for each $d \in \mathcal{D}$ and $\gamma < \mu$, $\lkakko \bigcup C_{d, \gamma} \mid d \in \mathcal{D}, \gamma < \mu\rkakko$ is a cofinal subset in $[\mu]^\kappa$ of size $\dmu$ and this contradicts the assumption $\dmu < \cf{[\mu]^\kappa, \subseteq}$. Therefore, there are $d\in \mathcal{D}$ and $\gamma < \mu$ such that $|\bigcup C_{d, \gamma}| > \kappa$.
    
    We pick such $d\in \mathcal{D}$ and $\gamma < \mu$. Let $\delta_0 = \min \lkakko \delta < \mu \mid \otp{\delta \cap \bigcup C_{d, \gamma}} = \kappa^+ \rkakko$. Recursively, we take $\left\langle x_\xi \in C_{d, \gamma} \mid \xi < \kappa\right\rangle$ satisfying $\sup(x_\xi \cap \delta_0) < \sup(x_\zeta \cap \delta_0)$ for all $\xi < \zeta < \kappa$. Now let $\alpha = \sup\{ \sup (x_\xi \cap \delta_0) \mid \xi < \kappa \}$. By the definition of $\pi$, we can find $\beta < \mu$ such that $\gamma' = \pi^{-1}(\alpha, \beta) \geq \gamma$. Thus $\lkakko f_{x_\xi}(\gamma') \mid \xi < \kappa \rkakko$ is unbounded in $\kappa$. However, $f_{x_\xi}(\gamma') \leq d(\gamma')$ for all $\xi < \kappa$. This is a contradiction.
\end{proof}

Note that it is easy to see that $\cf{[\mu]^\kappa, \subseteq} < \dmu$ is consistent since any $\kappa^+$-cc forcing preserves $\cf{[\mu]^\kappa, \subseteq}$. More concretely, starting with a model of $\GCH$, the forcing which adds $\mu^{++}$-many $\kappa$-Cohen reals forces $\dmu = \mu^{++} > \mu^+ = \cf{[\mu]^\kappa, \subseteq}$. In order to this, it suffices to show that $\kappa$-Cohen forcing adds an $V$-unbounded function of $\kappa^\mu$. Let $\langle \mu_i \mid i < \kappa \rangle$ be a continuous increasing sequence of $\mu$ with $\mu_0 = 0$, and $c \in \kappa^\kappa$ be a $\kappa$-Cohen real. Let us define $f\in \kappa^\mu$ by $f(\alpha) = c(i)$ if and only if $\alpha \in [\mu_i, \mu_{i +1} )$. Then such $f$ is a $V$-unbounded in $\kappa^\mu$.

Next, we apply this theorem to show the analogy of $\mathfrak{s} \leq \domega$. First, we introduce the generalization of $\mathfrak{s}$.

\begin{defi}
    Let $\lambda$ be a cardinal.
\begin{itemize}
    \item For $A, B \in [\lambda]^\lambda$, We say that $A$ \textit{splits} $B$ if $\left|B \cap A\right| = \left| B\setminus A \right| = \lambda$.
    \item For $\mathcal{S}\subseteq [\lambda]^\lambda$, we call $\mathcal{S}$ \textit{a splitting family} if for any $B \in [\lambda]^\lambda$, there is some $A \in \mathcal{S}$ such that $A$ splits $B$.
    \item We define 
        \begin{align*}
            \smu &= \min \lkakko \left| \mathcal{S} \right| \mid \mathcal{S} \text{ is a splitting family.} \rkakko
        \end{align*}
\end{itemize}
\end{defi}

In \cite{zapletal1997splitting}, Zapletal proved the following.

\begin{fac}\label{fac:zapletal fact}
    $\salephomega \leq \cf{[\aleph_\omega]^\omega, \subseteq}$.
\end{fac}

Below, we show that the extension of \Cref{fac:zapletal fact} to arbitrary singular cardinals is proved via a simple argument of constructing an appropriate Galois-Tukey map.
    
\begin{theo}\label{theo:smu <= cf(mu^kappa)}
    $\smu \leq \cf{[\mu]^\kappa, \subseteq}$.
\end{theo}

\begin{proof}
    Take a cofinal continuous increasing sequence $\langle \mu_i\mid i< \kappa \rangle$ in $\mu$ such that $\mu_i$ is regular for all $i < \kappa$. Let $C \subseteq [\mu]^\kappa$ be a cofinal subset in $([\mu]^\kappa, \subseteq)$ with $|C|= \cf{[\mu]^\kappa, \subseteq}$. For each $x \in C$, define $A_x=\lkakko \alpha \in x \mid \cf{\alpha} > \kappa\rkakko$. Then for each $\alpha \in A_x$, $\alpha \cap x$ is bounded in $\alpha$ and hence we can take some $\beta^x_\alpha < \alpha$ such that $\left( \beta^x_\alpha, \alpha \right)\cap x = \emptyset$. Note that $\lkakko \left( \beta^x_\alpha, \alpha \right) \relmiddle | \alpha \in A_x \rkakko$ is a pairwise disjoint family. We define $B_x \in \left[ \mu \right]^{\leq \mu}$ by 
    \[ B_x = \bigcup \lkakko \left( \beta^x_\alpha, \alpha \right) \relmiddle | \alpha \in A_x \land  \exists i < \kappa \left(\cf{\alpha} = \mu_{2i} \right) \rkakko. \]
    It suffices to show that $\lkakko B_x \mid x \in C \rkakko \cap [\mu]^\mu$ is a splitting family. 

    Let $A \in [\mu]^\mu$ and $\lkakko \gamma_\alpha \mid \alpha < \mu \rkakko$ be the increasing enumeration of $A$. For each $i < \kappa$, define $\delta_i = \sup\{\gamma_\alpha \mid \alpha < \mu_i \}$. Since $C$ is cofinal in $([\mu]^\kappa, \subseteq)$, there is some $x \in C$ such that $x\supseteq \lkakko \delta_i \mid i < \kappa \rkakko$. For each $i < \kappa$, $\cf{\delta_i} = \mu_i$ and $A\cap \delta_i$ is cofinal in $\delta_i$. Hence, $\left| A\cap  \left( \beta^x_{\delta_i}, \delta_i \right) \right| = \mu_i$. Since $\lkakko \left( \beta^x_\alpha, \alpha \right) \relmiddle | \alpha \in A_x \rkakko$ is pairwise disjoint,
            \begin{align*}
                \left| A\cap B_x \right| &\geq \left| \bigcup_{i < \kappa} A\cap \left( \beta ^x_{\delta_{2i}}, \delta_{2i} \right) \right| = \sum_{i < \kappa} \mu_{2i} = \mu \\
                \left| A\setminus B_x \right| &\geq \left| \bigcup_{i < \kappa} A\cap \left( \beta ^x_{\delta_{2i + 1}}, \delta_{2i + 1} \right) \right| = \sum_{i < \kappa} \mu_{2i + 1} = \mu.
            \end{align*}
    This completes the proof.
\end{proof}  

Then the following is straightforward.
  
\begin{cor}
    $\smu \leq \dmu$.
\end{cor}

\section{Separability of $\dmu$ and $2^\mu$}
In this section, we discuss the separability of $\dmu$ and $2^\mu$. First, we prepare basic facts about forcing.

\begin{defi}
    Let $\qposet \subseteq \pposet$ be posets, and $p \in \pposet$. We call $q_p\in \qposet$ \textit{a reduction of $p$} if $q \bot p$ implies $q\bot p_q$ for all $q\in \qposet$.
\end{defi}

\begin{fac}\label{fac:several facts}
Let $\mathbb{P}$ be a complete Boolean algebra, and $\qposet$ be a complete Boolean subalgebra of $\pposet$. Define $\dot{\rposet} = \lkakko \langle \check{p}, q \rangle \mid p \text{ is a reduction of } q \rkakko$.
\begin{enumerate}
    \item For any $p \in \pposet$, $\tilde{p} = \inf\lkakko q \in \qposet\mid p\leq q \rkakko$ is a reduction of $p$.
    \item Let $D = \lkakko \langle q, \check{p} \rangle \mid q \Vdash p \in \dot{\R} \rkakko$ be a dense subset of $\qposet \ast \dot{\rposet}$. Define the map $k:D \rightarrow \pposet$ by $k(q, \check{p}) = q \land p$. Then $k$ is a dense embedding. Furthermore, since $\pposet$ is complete, its natural completion
    \[ \bar{k}: \bigvee_{i} \langle q_i, \check{p_i} \rangle \mapsto \bigvee_{i} \bar{k} (q_i, \check{p_i} )\]
    is a dense embedding from $\qposet \ast \dot{\rposet}$ to $\pposet$.
    \item If $\kappa$ is regular and $\pposet$ is $\kappa$-cc, then $\Vdash_{\qposet} \dot{\rposet} \text{ is $\kappa$-cc}$.
\end{enumerate}
\end{fac}

First, we show that $\dmu$ and $2^\mu$ can be separated by a finite support iteration of the Cohen forcing in the case $\kappa \neq \omega$.

\begin{theo}\label{theo:main3}
    Assume $\GCH$ and $\kappa\neq \omega$. Let $\pposet$ be a finite support iteration of the Cohen forcing $2^{< \omega}$ of length $\mu^{++}$. Then $\Vdash_{\pposet} \dmu < 2^{\mu}$.
\end{theo}
\begin{proof}
    We show that $\left(\kappa^\mu\right)^V$ is a dominating family in $V^{\pposet}$. Let $\dot{f} \in V^\pposet$ be a $\pposet$-name such that $\Vdash_{\pposet} \dot{f} \in \kappa^\mu$. Since $\pposet$ is ccc, there exists a function $F \in ([\kappa]^{\aleph_0})^\mu \cap V$ such that $\Vdash_{\pposet} \dot{f}(\alpha) \in F(\alpha)$ for all $\alpha < \mu$. Therefore, $\Vdash_{\pposet} \dot{f} \leq \sup F$, and this implies $\Vdash_{\pposet} \dmu = \left| \left(\kappa^\mu\right)^V \right| = \mu^{+} < \mu^{++} = 2^{\aleph_0} \leq 2^{\mu}.$
\end{proof}

We do not know whether $\dmu < 2^\mu$ is consistent in the case $\kappa = \omega$. Below, we give some limitations on forcing $\dmu < 2^\mu$.

\begin{lem}\label{lem:existence of a good antichain}
    Let $\mathbb{P}$ be an atomless complete Boolean algebra, and $\qposet$ be a complete Boolean subalgebra of $\pposet$. Let $\dot{G}$ be the canonical name for generic filter of $\qposet$. If $\Vdash_\qposet \pposet / \dot{G} \text{ is non-trivial}$, then there is a countable antichain $A\subseteq \pposet$ such that $q \parallel a$ for all $a \in A$ and $q\in \qposet\setminus \lkakko \mathbf{0} \rkakko$.
\end{lem}

\begin{proof}
    Let $\lkakko \dot{r}_n \mid n < \omega \rkakko \subseteq V^{\qposet}$ be a set of $\qposet$-names such that, for all $m \neq n < \omega$, $\Vdash_\qposet \dot{r}_n \bot \dot{r_m}$. Such $\dot{r}_n$'s exist since $\Vdash_\qposet \pposet / \dot{G} \text{ is non-trivial}$. Let $\bar{k}: \qposet \ast \dot{\rposet} \rightarrow \pposet$ be a dense embedding mentioned in \Cref{fac:several facts}. Define $A=\lkakko \bar{k}( \langle \dot{r_n}, \mathbf{1} \rangle ) \mid n < \omega \rkakko$. We show that $A$ is as desired. Take an arbitrary $q \in \qposet\setminus \lkakko \mathbf{0} \rkakko$. Clearly, for all $n < \omega$, $\langle \check{\mathbf{1}}, q \rangle \parallel \langle \dot{r_n}, \mathbf{1}\rangle$ holds. Since $\bar{k}$ is a dense embedding, $\bar{k}(\langle \check{\mathbf{1}}, q \rangle) \parallel \bar{k}( \langle \dot{r_n}, \mathbf{1}\rangle)$ also holds. By the definition of $\bar{k}$, $\bar{k}(\langle \check{\mathbf{1}}, q \rangle) = \mathbf{1} \wedge q = q$. Therefore, $q \parallel \bar{k}( \langle \dot{r_n}, \mathbf{1}\rangle)$ and this completes the proof.
\end{proof}

\begin{theo}\label{theo:main4}
    Assume $\GCH$ and $\kappa = \omega$. Let $\lambda > \mu$ be a regular cardinal, and $\pposet$ be an atomless complete Boolean algebra with $|\pposet| = \lambda$. If $\pposet$ has $\mu$-cc and $\Vdash_\qposet \pposet / \dot{G} \text{ is non-trivial}$  for every complete subalgebra $\qposet\subseteq \pposet$ with $|\qposet| < |\pposet|$, then $\Vdash_{\pposet} \dmu =  \lambda = 2^\mu$.
\end{theo}
\begin{proof}
    We show that $\Vdash_{\pposet} \dmu > \nu$ for all $\nu < \lambda$. Let $\dot{\mathcal{F}}$ be a name for a family of functions in $\omega^\mu$ such that $\Vdash_{\pposet}|\dot{\mathcal{F}}| = \nu$. Then we can take nice names $\dot{f_\xi}$ for all $\xi < \nu$ such that $\Vdash_{\pposet} \dot{\mathcal{F}} = \lkakko \dot{f_\xi} \mid \xi < \nu \rkakko$. It suffices to show that $\Vdash_{\pposet} \dot{\mathcal{F}} \text{ is not a dominating family}$. Let $X$ be the set of all conditions in $\pposet$ which appear in $\dot{f_\xi}$ for some $\xi < \nu$. Note that $|X|< \lambda$ since $\mu,\nu < \lambda$.

    By induction on $\alpha < \mu$, we construct sequences $\langle \qposet_\alpha \mid \alpha < \mu \rangle$ and $\langle B_\alpha \mid \alpha < \mu \rangle$ such that
    \begin{itemize}
        \item For every $\alpha< \mu$, $\qposet_\alpha$ is a complete subalgebra of $\pposet$, 
        \item For every $\alpha< \mu$, $B_\alpha$ is a countable maximal antichain in $\pposet$, 
        \item $\qposet_0$ is the smallest complete subalgebra of $\pposet$ that contains $X$,
        \item $q\parallel a$ for each $\alpha < \mu$, $a\in B_\alpha$ and $q\in \qposet_\alpha\setminus \lkakko \mathbf{0} \rkakko$,
        \item $\qposet_{\alpha+1}$ is the smallest complete subalgebra of $\pposet$ that contains $\qposet_\alpha \cup B_\alpha$, and
        \item $\qposet_\delta$ is the smallest complete subalgebra of $\pposet$ that contains $\bigcup_{\alpha < \delta} \qposet_\alpha$ for limit $\delta$.
    \end{itemize}
    This construction is possible by \Cref{lem:existence of a good antichain} and the assumptions that $\lambda$ is regular and $\GCH$ holds. Let $\lkakko b_n^\alpha \mid n< \omega \rkakko$ be an injective enumeration of $B_\alpha$ for each $\alpha < \mu$. Define a $\pposet$-name $\dot{g}$ for a function in $\omega^\lambda$ that satisfies $b_n^\alpha\Vdash_{\pposet} \dot{g}(\alpha) = n$. We show that $\Vdash_{\pposet} \dot{g} \nleq \dot{f_\xi}$ for all $\xi < \nu$.

    For contradiction, suppose that there is $p \in \pposet$ such that $p\Vdash_{\pposet} \dot{g} \leq \dot{f_\xi}$ for some $\xi < \nu$. For each $\alpha < \mu$, let $p_\alpha = \inf\lkakko q \in \qposet_\alpha \mid p \leq q \rkakko \in \qposet_\alpha$. Then $\lkakko p_\alpha \mid \alpha < \mu \rkakko$ is a decreasing sequence in $\pposet$. 
    \begin{cla}
        There is some $\beta < \mu$ such that $p_{\beta + 1} = p_\beta$.
    \end{cla}
    \begin{proof}[Proof of claim]
        Suppose, for contradiction, that $p_{\beta + 1} \lneq p_\beta$ for all $\beta
         < \mu$. Then $\{ \neg p_{\beta + 1} \wedge p_\beta \mid \beta < \mu\}$ is an antichain with size $\mu$. This contradicts the assumption that $\pposet$ is $\mu$-cc.
    \end{proof}

    By the claim above, we can take an ordinal $\beta < \mu$ such that $p_{\beta + 1} = p_\beta$. By the choice of $B_{\beta}$, $p_{\beta}\parallel b$ for all $b \in B_\beta$. Since $p_{\beta} = p_{\beta + 1} \in \qposet_{\beta + 1}$ is a reduction of $p$ in $\qposet_{\beta + 1}$, $p \parallel b$ for all $B_\beta$. Let $A = \lkakko a_n \mid n < \omega \rkakko \subseteq X$ be a maximal antichain such that $a_n \Vdash_{\pposet} \dot{f_\xi} (\beta) = n$. 
    \begin{cla}
        $p \parallel a$ implies $p \wedge a \parallel b$ for all $a \in A$ and $b \in B_\beta$.
    \end{cla}
    \begin{proof}[Proof of claim]
        For contradiction, suppose that $p \parallel a$ and $p \wedge a \bot b$ for some $b \in B_\beta$.
        Then $p \bot a \wedge b$, and this implies $p_{\beta + 1} \bot a \wedge b$ since $a \wedge b \in \qposet_{\beta + 1}$ and $p_{\beta + 1}$ is a reduction of $p$ in $\qposet_{\beta + 1}$. Since $p_{\beta} = p_{\beta + 1}$, $p_{\beta}\bot a \wedge b$ and this follows $p_{\beta} \wedge a\bot b$. Since $p \parallel a$, $p_\beta \wedge a \neq 0$. However, $p_\beta \wedge a \in \qposet_{\beta}$ and this contradicts the fact that $q \parallel b$ for all $q \in \qposet_\beta\setminus \lkakko \mathbf{0} \rkakko$.
    \end{proof}
    Since $A$ is a maximal antichain, there is $n < \omega$ such that $p \parallel a_n$. Thus, by the claim above, $r= p \wedge a_n \wedge b_{n+1}^\beta \neq 0$, and so $r\Vdash_{\pposet} \dot{g} \nleq \dot{f_\xi}$. Since $r \leq p$, this is a contradiction.

    We show that $\Vdash_{\pposet} \lambda \geq 2^\mu$. Note that $V\models\GCH$ and $\pposet$ is a $\mu$-cc forcing. By counting the nice names for $2^\mu$, we can show that $\Vdash_{\pposet} \lambda \geq 2^\mu$. Since $\dmu \leq 2^\mu$ holds, $\Vdash_{\pposet} \lambda \leq \dmu \leq 2^\mu \leq \lambda$. This completes the proof.
\end{proof}

\section{Open questions}

One of the most interesting and important problems is the separability of $\dmu$ and $2^\mu$ when $\kappa=\omega$.
\begin{que}
    Is $\dmu < 2^\mu$ consistent when $\kappa = \omega$?
\end{que}

A related question is the following;

\begin{que}
    Does the consistency of $\dmu < 2^\mu$ with $\kappa = \omega$ imply the existence of some large cardinals? 
\end{que}

There remains room for a deep analysis of the lower bounds. For example,

\begin{que}
    Is $\mu^\kappa \leq \dmu$ provable in $\ZFC$ (with some natural additional assumptions)?
\end{que}

\section*{Acknowledgements}
    I would like to express my deepest gratitude to my supervisor Hiroshi Sakai for his patient supervision and inspiring discussions on my work. Without his guidance and help, this paper would have been difficult to complete.

        \bibliographystyle{IEEEtran}
        \bibliography{main}
\end{document}